\documentclass[9pt,reqno]{amsart}

 \usepackage{mathrsfs}

\usepackage{color}
\usepackage{amsmath, amsthm, amssymb}
\usepackage{amsfonts}
\usepackage[dvips]{epsfig}
\usepackage{graphicx}
\usepackage{caption}
\usepackage{subcaption}
\usepackage[english]{babel}
\usepackage[left=3cm,right=3cm,top=2cm,bottom=2cm]{geometry}
\usepackage{hyperref}

\usepackage{tikz}
\usepackage{rotating}
\usepackage[utf8]{inputenc}
\usepackage{cite}
\usepackage{amscd}
\usepackage{color}
\usepackage{bm}
\usepackage{enumerate}
\usepackage{verbatim}
\usepackage{hyperref}
\usepackage{amstext}
\usepackage{latexsym}
\theoremstyle{plain}
\newtheorem{theorem}{Theorem}[section]

\newtheorem{proposition}[theorem]{Proposition}
\newtheorem{lemma}[theorem]{Lemma}

\numberwithin{theorem}{section}
\numberwithin{equation}{section}

\newcommand{\average}{{\mathchoice {\kern1ex\vcenter{\hrule height.4pt
width 6pt depth0pt} \kern-9.7pt} {\kern1ex\vcenter{\hrule
height.4pt width 4.3pt depth0pt} \kern-7pt} {} {} }}

\def\R{\mathbb{R}}

\def\div{\text{div}}

\renewcommand{\b }{\beta }
\renewcommand{\d}{\delta }

\newcommand{\D }{\Delta }

\newcommand{\e }{\varepsilon }
\newcommand{\g }{\gamma}

\newcommand{\G }{\Gamma}
\renewcommand{\l }{\lambda }

\newcommand{\n }{\nabla }
\newcommand{\vp }{\varphi }
\renewcommand{\phi}{\varphi}

\newcommand{\rh }{\rho }

\newcommand{\s }{\sigma }

\renewcommand{\t }{\tau }

\renewcommand{\th }{\theta }

\renewcommand{\O }{\Omega }

\newcommand{\ov}{\overline}

\newcommand{\be}{\begin{equation}}
\newcommand{\ee}{\end{equation}}

\newcommand{\de}{\partial}

\newcommand{\ti}{\widetilde}

\renewcommand{\k}{\kappa}

\newcommand{\calC }{\mathcal{C}}

\newcommand{\calD }{\mathcal{D}}

\renewcommand{\H}{{\mathcal H}}

\newcommand{\cD}{{\mathcal D}}

\newcommand{\B}{{Q}}

\renewcommand{\epsilon}{\varepsilon}

\begin{document}
 
\date{\today}
\title[Influence of the curvature in the existence of solutions for a two Hardy-Sobolev critical exponents]
{ Influence of the curvature in the existence of solutions for a two Hardy-Sobolev critical exponents}
\author{El Hadji Abdoulaye THIAM}
\address{H. E. A. T. : Université Iba Der Thiam de Thies, UFR des Sciences et Techniques, département de mathématiques, Thies.}
\email{elhadjiabdoulaye.thiam@univ-thies.sn}
\author{Abdourahmane Diatta}
\address{A.D. : Université Assane Seck de Ziguinchor, UFR des Sciences et Technologies, département de mathématiques, Ziguinchor.}
\email{a.diatta20160578@zig.univ.sn}

\begin{abstract}
For $N\geq 4$, we let $\Omega$ be a bounded domain  of $\mathbb{R}^N$ and $\Gamma$ be  a closed curve contained in $\Omega$.  We study existence of positive solutions $u \in H^1_0\left(\Omega\right)$ to the  equation
\begin{equation}\label{Atusi}
-\Delta u+hu=\l\rho^{-s_1}_\Gamma u^{2^*_{s_1}-1}+\rho^{-s_2}_\Gamma u^{2^*_{s_2}-1} \qquad \textrm{ in } \Omega
\end{equation}
where $h$ is a continuous function and $\rho_\Gamma$ is the distance function to $\Gamma$.  We prove the existence of a mountain pass solution for this Euler-Lagrange equation depending on the local geometry of the curve and the potential $h$.
\end{abstract}
\maketitle
\section{Introduction}
For $N \geq 3$, the famous Caffarelli-Kohn-Nirenberg inequality asserts that: there exists a positive constant $C=C_{N, a, b}$ only depending on $N, a, b$, such that
\begin{equation}\label{CKNI}
C \left(\int_{\R^N} |x|^{-bq} |u|^q dx\right)^{2/q} \leq \int_{\R^N} |x|^{-2a} |\n u|^2 dx \qquad \forall u \in \calC^\infty_c(\R^N),
\end{equation}
where $N \geq 3$,  
$
-\infty < a< \frac{N-2}{2}, 0\leq b-a\leq 1
$
and 
$
q=\frac{2N}{N-2+2(b-a)},
$
see for instance \cite{CKN}. Note that the case $b=a+1$ and $p=2$, \eqref{CKNI} corresponds to the following Hardy inequality:
\begin{equation}\label{Hardy}
\displaystyle
\left(\frac{N-2}{2}\right)^2 \int_{\R^N} |x|^{-2} |u|^2 dx \leq \int_{\R^N} |\nabla u|^2 dx \qquad\forall u \in \mathcal{D}^{1,2}({\R^N}),
\end{equation}
where $\mathcal{D}^{1,2}({\R^N})$ denotes the completion of $\calC^\infty_c(\R^N)$ with respect to the norm 
$$
u \longmapsto \sqrt{\int_{\R^N} |\n u|^2 dx}.
$$
The constant $\left(\displaystyle\frac{N-2}{2}\right)^2 $ is sharp and never achieved in $\mathcal{D}^{1,2}({\R^N})$.
The case $a=b=0$ and $p=\frac{2N}{N-2}$corresponds to the famous Sobolev inequality:
\begin{equation}\label{Sobolev}
S_{N,0} \left(\int_{\R^N} |u|^{2^*} dx\right)^{2/2^*} \leq \int_{\R^N} |\nabla u|^2 dx \qquad \forall u \in  \mathcal{D}^{1,2}({\R^N}),
\end{equation}
where the best constant $$S_{N,0}= \displaystyle\frac{N(N-2)}{4}\omega_N^{2/N}$$ is achieved in $\mathcal{D}^{1,2}({\R^N})$. Here $\omega_N=|S^{N-1}|$  is the volume of the N-sphere and $ \displaystyle 2^*:=2^*(0)=\frac{2N}{N-2} $ is the critical Sobolev exponent.
By H\"{o}lder's inequality, we get the interpolation between the Hardy and the Sobolev inequalities, called Hardy-Sobolev inequality given by
\begin{equation}\label{Hardy-Sobolev}
S_{N,s} \left(\int_{\R^N} |x|^{-s} |u|^{2^*(s)} dx\right)^{2/2^*(s)} \leq \int_{\R^N} |\nabla u|^2 dx \qquad \forall u \in \mathcal{D}^{1,2}({\R^N}),
\end{equation}
where  for $s \in [0,2]$, we have $\displaystyle 2^*(s)= \frac{2(N-s)}{N-2}$ is the critical Hardy-Sobolev exponent. We refer to \cite{GK} for more details about Hardy-Sobolev inequality. The value of the best constant is
$$S_{N,s}:= (N-2)(N-s) \left[\frac{w_{N-1}}{2-s} \frac{\Gamma^2(N-\frac{s}{2-s})}{\Gamma(\frac{2(N-s)}{2-s})}\right]^{\frac{2-s}{N-s}},$$
where $\Gamma$ is the Gamma Euler function. It was computed by Lieb \cite{Lieb} when $s \in (0,2)$. The ground state solution is given, up to dilation, by
$$
w(x)=C_{N,s}(1+|x|^{2-s})^{\frac{2-N}{2-s}},
$$
for some positive known constant $C_{N,s}$.

The Caffarelli-Kohn-Nirenberg's inequality on domains and related problems have been studied these last years. For instance, we let $\Omega$ be a domain of $\R^N$ and consider the equation
\begin{equation}\label{ELE12}
\begin{cases}
-\div(|x|^{-2a} \n u)=|x|^{-bq} u^{q-1}, \quad u>0 \qquad &\textrm{ in $\Omega$}\\\
 u=0 &\textrm{ on $\de \O$}.
\end{cases}
\end{equation}
To study \eqref{ELE12}, one could let $w(x)=|x|^{-a} u(x)$. Direct computations show that
$$
\int_\O |x|^{-2a} |\n u|^2 dx =\int_{\O} |\n w|^2 dx-a(n-2-a) \int_{\O} |x|^{-2} w^2 dx.
$$
Then solutions of \eqref{ELE12} can be obtained by minimizing the following quotient
\begin{equation}\label{Inf-Min}
S^N_{a,b}(\O):= \inf_{u \in \calD^{1,2}_a(\O) \setminus \lbrace 0 \rbrace} \frac{\displaystyle\int_{\O} |\n w|^2 dx-a(n-2-a) \int_{\O} |x|^{-2} w^2 dx}{\left(\displaystyle\int_{\O} |x|^{-bq} |u|^q dx\right)^{2/q}},
\end{equation}
where $\calD^{1,2}_a(\O)$ be the completion of $\calC^\infty_c(\O)$ with respect to the norm
$$
u \longmapsto \sqrt{\int_\O |x|^{-2a} |\n u|^2 dx}.
$$
The question related to the attainability of the best constant $S^N_{a,b}(\O)$ in \eqref{Inf-Min} is studied by many authors. For more developments related to that, we refer the readers to \cite{BPZ, CH, CW,  CL, CC, DET, GK, GR, Li, Lieb, Lin, LW} and references therein.\\

When $0\in \de \O$, the existence of minimizers for $S^N_{a,b}(\O)$ was first studied by Ghoussoub-Kang \cite{GK} and Ghoussoub-Robert\cite{GR}. Later Chern and Lin \cite{CL} proved the existence of minimizer provided the mean curvature of the boundary at the origin is negative and ($a<b<a+1$ and $N\geq 3$) or ($b=a>0$ and $N\geq 4$).  The case $a=0$ and $0<b<1$ was first studied by \cite{GR} before the generalization in \cite{CL}. More generally questions related to Partial Differential Equations involving multiples Hardy-Sobolev critical exponents have been investigated these last decades. In particular, we let $\Omega$ be a domain of $\R^N$ such that $0\in \de \O$ and consider the equation
\begin{equation}\label{A1A100}
\begin{cases}
\displaystyle-\Delta u=\l\frac{u^{2^*_{s_1}-1}(x)}{|x|^{s_1}}+\frac{u^{2^*_{s_2}-1}}{|x|^{s_2}} \qquad & \textrm{ in $\O$}\\\
u(x)>0 & \textrm{ in $\O$},
\end{cases}
\end{equation}
where $0\leq s_2 <s_1<2$, $\l \in \R$  and for $i=1, 2$,  $2^*_{s_1}:=\frac{2(N-s_i)}{N-2}$ it the critical Hard-Sobolev exponent. When $s_2=0$ and $\l<0$, then equation \eqref{A1A100} has no nontrivial solution. For $\l>0$, $0<s_1<2$ and $s_2=0$, then using variational methods, Hsia Lin and Wadade \cite{HLW} proved existence of solutions provided  $N \geq 4$ and the mean curvature at the origin is negative. For the case $N=3$, $\l \in \R$ and $0<s_2<s_1<2$, the equation \eqref{A1A100} has a least-energy solution provided the mean curvature at the origin is negative, see \cite{LLin}.

Concerning the existence and non-existence of solution related to equation \eqref{A1A100} in the half-space $\O=\R^N_+$, we refer to Bartsch-Peng and Zhang \cite{BPZ} for the case $0<s_2<s_1=2$ and 
$
\l<\left(\frac{N-2}{2}\right)^2
$;
to Musina \cite{Musina} when $N\geq 4$, $s_2=0$, $s_1=2$ and $0<\l <\left(\frac{N-2}{2}\right)^2$ and to Hsia, Lin and Wadade \cite{HLW} when $s_2=0$, $0<s_1<2$ and $\l >0$.\\
%
%
%
%
%
%
%
%
%
%
%
%
%
%
%
%
%
%
%
%
%
%
%
%
%

In this paper, we are concerned with the effect of the local geometry of the singularity $\Gamma$ in the existence of solutions of the following non-linear partial differential equation involving two Hardy-Sobolev critical exponents. More precisely, letting $h$ be a continuous function and $\l$ be a real parameter, we consider
\begin{equation}\label{Euler-Lagrange11}
\begin{cases}
\displaystyle-\Delta u(x)+ h u(x)=\l \frac{u^{2^*_{s_1}-1}(x)}{\rho_\G^{s_1}(x)}+\frac{u^{2^*_{s_2}-1}(x)}{\rho_\G^{s_2}(x)} \qquad & \textrm{ in $\O$}\\\\
u(x)>0 \qquad \textrm{ and } \qquad u(x)=0 &\textrm{ on $\de \Omega$},
\end{cases}
\end{equation}
where $\rho_\G(x):=\inf_{y \in \Gamma}|y-x|$ is the distance function to the curve $\Gamma$,  $0< s_2 <s_1<2$, $2^*_{s_1}:=\frac{2(N-s_1)}{N-2}$ and $2^*_{s_2}:=\frac{2(N-s_2)}{N-2}$ are two critical Hardy-Sobolev exponents.
To study the equation  \eqref{Euler-Lagrange11}, we consider the following non-linear functional $\Psi: H^1_0(\O) \to \R$ defined by:
\begin{equation}\label{Functional}
\Psi(u):=\frac{1}{2} \int_\O |\n u|^2 dx+\frac{1}{2} \int_\O h(x) u^2 dx- \frac{\l}{2^*_{s_1}} \int_\O \frac{|u|^{2^*_{s_1}}}{\rho_\G^{s_1}(x)} dx-\frac{1}{2^*_{s_2}} \int_\O \frac{|u|^{2^*_{s_2}}}{\rho_\G^{s_2}(x)} dx.
\end{equation}
It is easy to verify that there exists a positive constant $r>0$ and $u_0 \in H^1_0(\Omega)$ such that $\|u_0\|_{H^1_0(\O)}>r$ and
$$
\inf_{\|u\|_{H^1_0(\O)}=r} \Psi(u) >\Psi(0) \geq \Phi(u_0), 
$$
see for instance Lemma \ref{c1A} below.
Then the point $(0, \Psi(0))$ is separated from the point $(u_0, \Psi(u_0))$ by a ring of mountains. Set
\begin{equation}\label{Heat}
c^*:=\inf_{P \in \mathcal{P}} \max_{v \in P} \Psi(v),
\end{equation}
where $\mathcal{P}$ is the class of continuous paths in $H^1_0(\O)$ connecting $0$ to $u_0$. Since $2^*_{s_2}>2^*_{s_1}$, the function $t \longmapsto \Psi(tv)$ has the unique maximum for $t \geq 0$. Furthermore, we have
$$
c^*:=\inf_{u \in H^1_0(\O), u \geq 0, u \neq 0} \max_{t \geq 0} \Psi(tu).
$$
Due to the fact that the embedding of $H^1_0(\O)$ into the weighted Lebesgue spaces $L^{2^*_{si}}(\rho_\Gamma^{-si} dx)$ is not compact, the functional $\Psi$  does not satisfy the Palais-Smale condition. Therefore, in general $c^*$ might not be a critical value for $\Psi$.

To recover compactness, we study the following non-linear problem: let $x=(y, z) \in \R\times \R^{N-1}$ and consider 
\begin{equation}\label{A1A1}
\begin{cases}
\displaystyle-\Delta u=\l\frac{u^{2^*_{s_1}-1}(x)}{|z|^{s_1}}+\frac{u^{2^*_{s_2}-1}}{|z|^{s_2}} \qquad & \textrm{ in $\R^N$}\\\
u(x)>0 & \textrm{ in $\R^N$}.
\end{cases}
\end{equation}
To obtain solutions of \eqref{A1A1}, we consider the functional $\Phi: \calD^{1,2}(\R^N)$ defined by
$$
\Phi(u):=\frac{1}{2} \int_{\R^N} |\n u|^2 dx- \frac{\l}{2^*_{s_1}} \int_{\R^N} |z|^{-s_1}|u|^{2^*_{s_1}} dx-\frac{1}{2^*_{s_2}} \int_{\R^N} |z|^{-s_2}|u|^{2^*_{s_2}} dx.
$$
Next, we define
$$
\beta^*:=\inf_{u \in D^{1, 2}(\R^N), u \geq 0, u \neq 0} \max_{t \geq 0} \Phi(tu).
$$
Then we get compactness provided
$$
c^*<\beta^*,
$$
see Proposition \ref{Prop-Compactness} below. So it is important to study existence, symmetry and decay estimates of non-trivial solution $w\in \calD^{1,2}(\R^N)$ of \eqref{A1A1}. Then we have the following results.
\begin{theorem}\label{TheoremA}
Let $N \geq 3$, $0 \leq s_2<s_1<2$, $\l \in \R$. Then equation
\begin{equation}\label{Euler-Lagrange1}
\begin{cases}
\displaystyle-\Delta u=\l\frac{u^{2^*_{s_1}-1}(x)}{|z|^{s_1}}+\frac{u^{2^*_{s_2}-1}}{|z|^{s_2}} \qquad & \textrm{ in $\R^N$}\\\
u(x)>0 & \textrm{ in $\R^N$}
\end{cases}
\end{equation}
has a positive ground state solution $w \in \calD^{1,2}(\R^N)$. Moreover $w$ depend only on $y$ and $|z|$. In other words, there exists a function $\theta: \R\times \R_+ \to \R_+$ such that
$$
w(x)=\theta(y, |z|).
$$
\end{theorem}
Next we have the following decay estimates of the solution $w$ and its higher order derivatives.
\begin{theorem}\label{TheoremB}
Let $w$ be a solution of the Euler-Lagrange equation  \eqref{Euler-Lagrange1}. Then
\begin{itemize}
\item[(i)] there exists two positive constants $c_1<c_2$ such that:
$$
\frac{c_1}{1+|x|^{N-2}}\leq u(x) \leq \frac{c_2}{1+|x|^{N-2}}, \qquad \forall x\in \R^N.
$$
\item[(ii)] For   $|x|= |(t,z)|\leq 1$  
$$ 
|\n w (x)|+ |x| |D^2 w (x)|\leq C_2 |z|^{1-s_1}
$$
\item[(iii)] For   $|x|= |(t,z)|\geq 1$ 
$$
|\n w(x)|+ |x| |D^2 w(x)|\leq C_2 \max(1, |z|^{-s_1})|x|^{1 -N}.
$$
\end{itemize}
\end{theorem}
These two theorems will play a crucial role in the following which is our main result.
Then we have
\begin{theorem}\label{th:main1}
Let $N \geq 4$, $0\leq s_2 <s_1<2$ and   $\O$  be a   bounded domain of $\R^N$. Consider  $\G$ a smooth closed curve contained in $\O$.
Let  $h$ be  a continuous function such that the linear operator $-\D+h$ is coercive. Then there exists a positive constant $A^N_{s_1,s_2}$, only depending on $N$, $s_1$ and $s_2$ with the property that if  there exists  $y_0\in \G$  such that 
\begin{equation}\label{eq:h-bound-main-th-1}
A^N_{s_1, s_2} |\k (y_0)|^2+h(y_0) <0  \qquad\textrm{ for $N \geq 4$},
\end{equation}
then $c^*<\b^*$,
where $\k:\G\to \R^N$ is the curvature vector of   $\G$. 
Moreover there exists $u \in H^1_0(\Omega)\setminus \lbrace 0\rbrace$ non-negative solution of
$$
-\Delta u(x)+ h u(x)=\l \frac{u^{2^*_{s_1}-1}(x)}{\rho_\G^{s_1}(x)}+\frac{u^{2^*_{s_2}-1}(x)}{\rho_\G^{s_2}(x)} \qquad \textrm{ in $\O$}.
$$
\end{theorem}

The effect    of curvatures in the study of   Hardy-Sobolev  inequalities have been intensively studied in the recent years.  For each of these works, the sign of the curvatures at the point of singularity plays important roles for the existence   a solution.  The first paper, to our knowledge, being the one  of Ghoussoub and Kang \cite{GK} who considered the Hardy-Sobolev inequality with singularity at the boundary.  For more  results in this direction,   see the works of    Ghoussoub and Robert in \cite{GR2,GR3,GR4, GR51}, Demyanov and Nazarov \cite{DN}, Chern and Lin \cite{CL}, Lin and Li \cite{LLin}, the first
 author, Fall and Minlend in \cite{FMT} and the references there in.
The Hardy-Sobolev inequality with interior singularity on Riemannian manifolds have been studied by Jaber \cite{Jaber1} and the first author \cite{Thiam-Hardy}. Here also the impact of the scalar curvature at the point singularity plays an important role for the existence of minimizers in higher dimensions $N\geq 4$.   The paper  \cite{Jaber1} contains also existence result under positive mass condition for $N=3$. We point out theat the 3-dimensional version of this paper is presented in \cite{Thiam-Mass-Effect}. The existence of  solution does not depends on the local geometry of the singularity but on the regular part of the Green function of the operator $-\Delta+h$.\\

The proof of Theorem \ref{th:main1} relies on test function methods. Namely to build appropriate test functions allowing to compare $c^*$ and $\b^*$.  While it  always holds that $c^*\leq \b^*$, our main task is  to find a function for which  $c^*<\b^*$, see Section \ref{Section4}. This then allows to recover compactness and thus every minimizing sequence  for $ c^*$ converges to   a minimizer, up to a subsequence. Building these approximates solutions requires to have sharp decay estimates of a  minimizer $w$ for $\b^*$, see Section \ref{Section3}. Section \ref{Section2} is devoted to the  local parametrization and computation of the local metric.
\section{Proof of Theorem \ref{TheoremA} and Theorem \ref{TheoremB}}\label{Section3}
\begin{theorem}
Let $N \geq 3$, $x:=(y, z) \in \R\times \R^{N-1}$, $0<s_2<s_1<2$ and $\l\in \R$. Then there exists $w\in \calD^{1,2}(\R^N)$ positive, satisfying
\begin{equation}\label{Euler-Lagrange-Equation}
-\D w=\l \frac{w^{2^*_{s_1}-1}}{|z|^{s_1}}+\frac{w^{2^*_{s_2}-1}}{|z|^{s_2}} \qquad \textrm{ in } \R^N.
\end{equation}
\end{theorem}
\begin{proof}
Applying similar arguments as in [\cite{LLin}, Theorem 1.2], it is easy to prove this result. So we omit the proof.
\end{proof}
%
%
%
%
%
%
%
%
%
Next we will establish symmetry and decay estimates properties of positive solutions $u \in \calD^{1,2}(\R^N)$ of the following Euler-Lagrange equation \eqref{Euler-Lagrange-Equation}. Rewritten equation \eqref{Euler-Lagrange-Equation} as follows
$$
-\D u=\frac{f(x)}{|z|^{s_1}} u+ \frac{g(x)}{|z|^{s_1}},
$$
where $f, g \in L^P_{loc}(\R^N)$ for some $p> \frac{N}{2-s_1}$, then we have
\begin{proposition}\label{Regularity}
Let $u$ is a solution of the Euler-Lagrange equation \eqref{Euler-Lagrange-Equation}. We assume that
$$
\begin{cases}
\displaystyle s_1< 1+\frac{1}{N}\qquad&\textrm { if  $N \geq 4$ }\\\
\displaystyle s_1<\frac{3}{2} &\textrm{ if  $N=3$}.
\end{cases}
$$
Then $u \in \calC^\infty$ in the $y$ variable while, in the $z$ variable, it is $\calC^{1, \alpha}$ for all $\alpha<1-s_1$ if $s_1<1$ and $\calC^{0, \alpha}(\R^N)$ for all $\alpha <2-s_1$ if $1\leq s_1 <2$.
\end{proposition}
This result is  due to Fabbri-Mancini-Sandeep[\cite{FMS}, Lemma 3.2 and Lemma 3.3]. 
This then allows  the following symmetry and decay estimates result.
\begin{proposition}\label{Prop-Decay-Esti1}
Let $u$ be a solution of the Euler-Lagrange equation \eqref{Euler-Lagrange-Equation}. Then
\begin{itemize}
\item[(i)] the function $u$ depends only on $y$ and $|z|$\\
\item[(ii)] there exists two constants $0<c_1<c_2$ such that:
\begin{equation}\label{Decay-Est1}
\frac{c_1}{1+|x|^{N-2}}\leq u(x) \leq \frac{c_2}{1+|x|^{N-2}}, \qquad \forall x\in \R^N.
\end{equation}
\end{itemize}
\end{proposition}
\begin{proof}
The proof of the symmetry is based on the moving plane method, see for instance \cite{AAZ, BNAZ, DNAZ, CLAZ, SAZ} and references therein.
\end{proof}
We close this section by proving the following decay properties of $w$ involving its higher derivatives.%
\begin{proposition}\label{Prop-Decay-Esti2}
Let $w$ be a solution of \eqref{Euler-Lagrange-Equation}.hen there exist positive constant $C$, only depending on $N$ and $s_1$ and $s_2$, such that
\item[(ii)] For   $|x|= |(t,z)|\leq 1$  
$$ 
|\n w (x)|+ |x| |D^2 w (x)|\leq C_2 |z|^{1-s_1}
$$
\item[(iii)] For   $|x|= |(t,z)|\geq 1$ 
$$
|\n w(x)|+ |x| |D^2 w(x)|\leq C_2 \max(1, |z|^{-s_1})|x|^{1 -N}.
$$
\end{proposition}
\begin{proof}
Let $\theta: \R_+ \times \R_+ \to \R_+$ be a function such that
$$
w(x)=w(y, z)= \theta(y, |z|).
$$
Using polar coordinates, the function $\theta=\theta(t, \rho)$ verifies 
\be\label{eq:eq-satis-th}
\rho^{2-N} (\rho^{N-2} \th_2)_2+\th_{11}= \l\rho^{-s_1} \th^{2^*_{s_1}-1}+\rho^{-s_2} \th^{2^*_{s_2}-1} \qquad\textrm{ for $t,\rh\in \R_+$},
\ee
where $\theta_1$ and $\theta_2$ are respectively the derivatives of $\theta$ with respect to the first and the second variables. Then integrating this identity in  the $\rho$ variable, we therefore get, for every $\rho>0$,
\begin{align*}
\th_2(t,\rho)=-\frac{1}{\rho^{N-2}}\int_0^\rho  r^{N-2}  \th_{11} (t,r)dr&+\frac{\l}{\rho^{N-2}}\int_0^\rho r^{N-2} r^{-s_1} \th^{2^*_{s_1}-1} (t,r) dr\\\\
&\frac{1}{\rho^{N-2}}\int_0^\rho r^{N-2} r^{-s_2} \th^{2^*_{s_2}-1} (t,r) dr.
\end{align*}
Next differentiating with respect to the first variable, we get 
\begin{align*}
\th_{12}(t,\rho)=\frac{-1}{\rho^{N-2}}\int_0^\rho  r^{N-2}  \th_{111} (t,r)dr&+\frac{\l}{\rho^{N-2}}\int_0^\rho r^{N-2} r^{-s_1} \th_1(t,r) \th^{2^*_{s_1}-2} (t,r) dr\\\\
&+\frac{1}{\rho^{N-2}}\int_0^\rho r^{N-2} r^{-s_2} \th_1(t,r) \th^{2^*_{s_2}-2} (t,r) dr.
\end{align*}
By  Proposition \ref{Regularity} and the fact that  $2^*_{s_2}> 2^*_{s_1}\geq 2$, we obtain
\begin{equation}\label{As123}
|\th_2(t,\rho)|+ |\th_{12}(t,\rho)| \leq  C\left( \rho+\rho^{1-s_1}+\rho^{1-s_2}\right)  \leq C \rho^{1-s_1} \qquad\textrm{ for $|(t,\rho)|\leq 1$}.
\end{equation}
Now using this in \eqref{eq:eq-satis-th}, we get 
\begin{equation}\label{Max1}
|\th_{22}|\leq C \rh^{-s_1}, \quad \textrm{ for $|(t,\rho)|\leq 1$}.
\end{equation}
By \eqref{As123} and \eqref{Max1}, we obtain
$$
|\th_2(t,\rho)|+ |\th_{12}(t,\rho)|+\rho|\th_{22}|\leq  C\rho^{1-s_1}.
$$
Therefore, it easy follows that  
$$ 
|\n w (x)|+ |x| |D^2 w (x)|\leq C_2 |z|^{1-s_1}, \qquad \textrm{ for all $|x|= |(t,z)| \leq 1$}
$$
and for   $|x|= |(t,z)|\geq 1$ that
$$
|\n w(x)|+ |x| |D^2 w(x)|\leq C_2 \max(1, |z|^{-s_1})|x|^{1 -N}.
$$
This then completes the proof.
\end{proof}
\section{Local Parametrization and metric}\label{Section2}
Let    $\G\subset \R^N$ be  a smooth closed  curve. Let $(E_1;\dots; E_N)$ be an orthonormal basis of $\R^N$.
 For $y_0\in \G$ and $r>0$ small,  we consider the curve $\gamma:\left(-r, r\right) \to \G$,   parameterized by arclength such that $\gamma(0)=y_0$. Up to a translation and a rotation,  we may assume that $\g'(0)=E_1$.     We choose a smooth   orthonormal frame field $\left(E_2(y);...;E_N(y)\right)$ on the normal bundle of $\G$ such that $\left(\g'(y);E_2(y);...;E_N(y)\right) $ is an oriented basis of $\R^N$ for every $y\in (-r,r)$, with $E_i(0)=E_i$. \\
We fix the following notation, that will be used  a lot in the paper,
$$
 Q_r:=(-r,r)\times B_{\R^{N-1}}(0,r) ,
$$
where $B_{\R^k}(0,r)$ denotes the ball in $\R^k$ with radius $r$ centered at the origin.
 Provided $r>0$  small, the map $F_{y_0}: Q_r\to \O$, given by 
$$
 (y,z)\mapsto  F_{y_0}(y,z):= \gamma(y)+\sum_{i=2}^N z_i E_i(y),
$$
is smooth and parameterizes a neighborhood of $y_0=F_{y_0}(0,0)$. We consider $\rho_\G:\G\to \R$ the distance function to the curve given by 
$$
\rho_\G(y)=\min_{\ov y\in \Gamma}|y-\ov y|.
$$
In the above coordinates, we have 
\begin{equation}\label{eq:rho_Gamm-is-mod-z}
\rho_\G\left(F_{y_0}(x)\right)=|z| \qquad\textrm{ for every $x=(y,z)\in Q_r.$} 
\end{equation}
Clearly, for every $t\in (-r,r)$ and $i=2,\dots N$, there  are real numbers $\k_i(y)$ and $\tau^j_i(y)$  such that
\begin{equation}\label{DerivE_i}
E_i^\prime(y)= \kappa_i(y) \gamma^\prime(y)+\sum_{ {j=2} }^N\tau_i^j(y) E_j(y).
\end{equation}
The quantity  $\kappa_i(y)$ is the curvature in the $E_i(y)$-direction while $\tau_i^j(y)$ is the torsion from the osculating plane spanned by $\{\g'(y); E_j(y)\}$ in the direction  $E_i $.  We note that provided $r>0$ small, $\k_i$ and $\t_i^j$ are smooth functions on $(-r,r)$. Moreover, it is easy to see that 
\be\label{eq:tau-antisymm}
\t^j_i(y)=-\t^i_j(y) \qquad\textrm{ for $i,j=2,\dots, N$.   } 
\ee
The curvature vector is $\k:\G\to \R^N$  is defined as  $\k(\g(y)):=\sum_{i=2}^N \k_i(y) E_i(y)$ and its norm is given by  $|\k\g(y)|:=\sqrt{\sum_{i=2}^N\k^2_i(y)}$.
Next, we derive the expansion of the metric induced by the parameterization $F_{y_0}$ defined above.
For $x=(y,z) \in Q_r$, we define 
$$
g_{11}(x)=   {\de_y F_{y_0}}(x) \cdot   {\de_y F_{y_0}} (x) , \quad g_{1i}(x)=   {\de_y F_{y_0}} (x)\cdot   {\de_{z_i} F_{y_0}}(x) ,\quad  g_{ij}(x)=   {\de_{z_j} F_{y_0}}(x) \cdot   {\de_{z_i} F_{y_0}}(x).
$$
We have the following result.
\begin{lemma}\label{MaMetric}
There exits $r>0$, only depending on $\G$ and $N$, such that  for ever $x=(t,z)\in Q_r$ 
\begin{equation}\label{Vert}
\begin{cases}
\displaystyle g_{11}(x)=1+2\sum_{i=2}^Nz_i \kappa_i(0)+2y\sum_{i=2}^Nz_i \kappa^\prime_i(0)+\sum_{ij=2}^N z_i z_j \kappa_i(0) \kappa_j(0)+\sum_{ij=2}^N z_i z_j  \b_{ij}(0)+O\left(|x|^3\right)\\\
\displaystyle g_{1i}(x)=\sum_{j=2}^N z_j \tau^i_j(0)+y\sum_{j=2}^N z_j \left(\tau^i_j\right)^\prime(0)+O\left(|x|^3\right)\\\
\displaystyle g_{ij}(x)=\d_{ij},
\end{cases}
\end{equation}
where  $\b_{ij}(y):=\sum_{l=2}^N \tau_i^l(y) \tau_j^l(y).$
\end{lemma}
As a consequence we have the following result.
\begin{lemma}\label{MaMetricMetric}
There exists $r>0$ only depending on $\G$ and $N$, such that for every $x\in Q_r$, we have 
\begin{equation}\label{DeterminantMetric}
\sqrt{|g|}(x)=1+\sum_{i=2}^N z_i \kappa_i(0)+y\sum_{i=2}^N z_i \kappa_i^\prime(0)+\frac{1}{2}\sum_{ij=2}^N z_i z_j \kappa_i(0) \kappa_j(0)+O\left(|x|^3\right),
\end{equation}
where  $|g|$ stands for the determinant of $g$.
Moreover  $g^{-1}(x)$, the matrix  inverse   of $g(x)$,   has components given by
\begin{equation}\label{InverseMetric}
\begin{cases}
\displaystyle g^{11}(x)=1-2\sum_{i=2}^N z_i \kappa_i(0)-2y\sum_{i=2}^N z_i \kappa_i^\prime(0)+3\sum_{ij=2}^N z_i z_j \kappa_i(0) \kappa_j(0)+O\left(|x|^3\right)\\
\displaystyle g^{i1}(x)=-\sum_{j=2}^N z_j \tau^i_j(0)-y\sum_{j=2}^N z_j \left(\tau^i_j\right)^\prime(0)+2\sum_{j=2}^N z_l z_j \kappa_l(0) \tau^i_j(0)+O\left(|x|^3\right)\\
\displaystyle g^{ij}(x)=\d_{ij}+\sum_{lm=2}^N z_l z_m \tau^j_l(0) \tau^i_m(0)+O\left(|x|^3\right).
\end{cases}
\end{equation}
\end{lemma}
We will also need the following estimates result.
\begin{lemma}\label{lem:to-prove}
Let $v\in \cD^{1,2}(\R^N)$, $N\geq 3,$ satisfy $v(y,z)=\ov\th(|y|,|z|)$, for some some function $\ov\th:\R_+\times \R_+\to \R$. Then  for $0<r<R$, we have 
\begin{align*}
 \int_{\B_{R}\setminus \B_{r}}|\n v|^2_{g}\sqrt{|g|} dx&= \int_{\B_{R}\setminus \B_{r}}|\n v|^2  dx+  \frac{|\kappa(x_0)|^2}{N-1} \int_{\B_{R}\setminus \B_{r}} |z|^2 \left|{\de_y v} \right|^2 dx\\
&+ \frac{|\kappa(x_0)|^2}{2(N-1)} \int_{\B_{R}\setminus  \B_{r}} |z|^2 |\n v|^2 dx
+O\left( \int_{\B_{R}\setminus \B_{r}} |x|^3 |\n v|^2 dx \right).
\end{align*}
\end{lemma}
For the proofs of these Lemma \ref{MaMetric}, Lemma \ref{MaMetricMetric} and Lemma \ref{lem:to-prove}, we refer to the paper of the first author and Fall \cite{Fall-Thiam}. See also \cite{Th1} for a generalization.
\section{Existence Result in domains}\label{Section5}
The aim of this section is to prove the following result.
\begin{proposition}\label{th:main12}
Let $N \geq 4$, $0\leq s_2 <s_1<2$ and   $\O$  be a   bounded domain of $\R^N$. Consider  $\G$ a smooth closed curve contained in $\O$.
Let  $h$ be  a continuous function such that the linear operator $-\D+h$ is coercive. We assume that
\begin{equation}\label{eq:h-bound-main-th-1}
c^*:=\sup_{t \geq 0} \Psi(v)< \b^*.
\end{equation}
Then there exists a positive function $u \in H^1_0(\O)$ solution of the Euler-Lagrange equation
\begin{equation}\label{HHH1}
-\Delta u+hu=\l\rho^{-s_1}_\Gamma u^{2^*_{s_1}-1}+\rho^{-s_2}_\Gamma u^{2^*_{s_2}-1} \qquad \textrm{ in } \Omega.
\end{equation}
\end{proposition}
The proof of Proposition \ref{th:main12} is divided into various preliminaries results. We start by the following.
\begin{lemma}\label{OuiOui}
For $N \geq 3$, we let $\O$ be an open subset of $\R^N$ and let $\G\subset \O$ be a smooth closed curve contained in $\Omega$.
Then for every $r>0$, there exists $  c_r>0$, only  depending on $\O,\G,N,\s$ and $r$, such that  for every  $u \in H^1_0(\O)$  
$$
\left(\frac{1}{2}-\frac{1}{2^*_\s}\right) \int_\O |\n u|^2 dx+\left(\frac{1}{2^*_{s_1}}-\frac{1}{2^*_{s_2}}\right) \int_\O \rho_\Gamma^{-s_2} dx+c_r \int_\O | u|^2 dy\geq c^*.
$$
where, for $0\leq s_2 <s_1<2$, $2^*_{s_1}=\frac{2(N-s_1)}{N-2}$ and $2^*_{s_2}=\frac{2(N-s_2)}{N-2}$.
\end{lemma}
\begin{proof}
We let $r>0$ small. We can cover a tubular neighborhood of $\G$ by a finite number of  sets $\left(T^{y_i}_r\right)_{1\leq i \leq m}$  given by 
$$
T^{y_i}_r:=F_{y_i}\left(Q_r\right), \qquad \textrm{ with $y_i\in\G$}.
$$
We refer to  Section \ref{Section2} for the parameterization  $F_{y_i}:Q_r\to \O$. 
Let $\left(\vp_i\right)_{1\leq i \leq m}$ be  a   partition of unity subordinated to this covering such that
\begin{equation}\label{eq:Part-unity}
\sum_i^{m } \vp_i  =1 \qquad \textrm{and} \qquad |\n \vp_i ^{\frac{1}{2^*_\s}} |\leq C \qquad \textrm{ in } U:=\displaystyle \cup_{i=1}^{m} T^{y_i}_r,
\end{equation}
for some positive constant $C$.
We define
\be \label{eq:def-psi_i}
\psi_i(y):=\vp_i^{\frac{1}{2^*_\s}} (y) u(y) \qquad \textrm{ and } \qquad \ti{\psi_i}(x)=\psi_i (F_{y_i}(x)).
\ee
Then, we have
\begin{align} \label{App1}
\displaystyle\int_\O \rho^{-\s}_\G |u|^{2^*_\s} dy \geq   \int_U \rho^{-\s}_\G \left| u \right|^{2^*_\s} dy=\sum_i^{m} \int_{T^{y_i}_r} \rho^{-\s}_\G  |\psi_i|^{2^*_\s} dy.
\end{align}
By  change of variables and Lemma \ref{MaMetricMetric}, we have
\begin{align}\label{App2}
\int_{T^{y_i}_r} \rho^{-\s}_\G |\psi_i|^{2^*_\s} dy=\int_{Q_r} |z|^{-\s} |\ti{\psi}_i|^{2^*_\s} \sqrt{|g|}(x) dx\geq \left(1-c r\right) \int_{Q_r} |z|^{-\s} |\ti{\psi}_i|^{2^*_\s}  dx,
\end{align}
for some positive constant $c$.
By \eqref{App1} and \eqref{App2} and the summing over $i=1, \cdots, m$, we obtain
\begin{equation}\label{Expo1}
\int_\O \rho^{-\s}_\G |u|^{2^*_\s} dy \geq \left(1-c r\right) \sum_{i=1}^m \int_{Q_r} |z|^{-\s} |\ti{\psi}_i|^{2^*_\s}  dx=(1-cr)\int_U  |z|^{-\s} |\tilde{u}(x)|^{2^*_\s} dx,
\end{equation}
with $\tilde{u}:=u(F_{y_i}(x))$.
Next, we have
\begin{equation}\label{App3}
\int_\O |\n u|^2 dx\geq \int_U | \n u|^2 dy=\sum_i^{m} \int_{T^{y_i}_r} |\n \psi_i|^2dy.
\end{equation}
By  change of variables, Lemma \ref{MaMetricMetric}, \eqref{eq:Part-unity} and \eqref{eq:def-psi_i}, we have
\begin{align*}
\int_{T^{y_i}_r} |\n \psi_i|^2dy&=\int_{Q_r} |\n \ti{\psi}_i|^2\sqrt{|g|}(x) dx\geq \left(1-c r\right) \int_{Q_r}|\n \ti{\psi}_i|^2 dx\\\\
&\geq \left(1- c' r\right) \int_{T^{y_i}_r} |\n (\phi_i^{\frac{1}{2^*_\s}}  u)|^2 dy=\int_{T^{y_i}_r} \phi_i^{\frac{2}{2^*_\s}}|\n \tilde{u}|^2 dy-c_r \int_\O | u|^2 dy,
\end{align*}
for some positive constants $c$ and $c_r$. Therefore
\begin{equation}\label{App4}
\int_{T^{y_i}_r} |\n \psi_i|^2dy \geq \int_{T^{y_i}_r} |\n \tilde{u}|^2 dy-c_r \int_\O | u|^2 dy.
\end{equation}
Hence combining \eqref{App3} and \eqref{App4}, we obtain
\begin{equation}\label{Expo2}
\int_\O |\n u|^2 dx \geq \int_{U} |\n \tilde{u}|^2 dy-c_r \int_\O | u|^2 dy.
\end{equation}
Thanks to \eqref{Expo1},\eqref{Expo2} and the definition of $\b^*$, we get
\begin{align*}
\left(\frac{1}{2}-\frac{1}{2^*_\s}\right) \int_\O |\n u|^2 dx&+\left(\frac{1}{2^*_{s_1}}-\frac{1}{2^*_{s_2}}\right) \int_\O \rho_\Gamma^{-s_2} dx+c_r \int_\O | u|^2 dy \\\\
&\geq  \left(\frac{1}{2}-\frac{1}{2^*_\s}\right) \int_{U} |\n \tilde{u}|^2 dy+\left(\frac{1}{2^*_{s_1}}-\frac{1}{2^*_{s_2}}\right) (1-cr)\int_U |\tilde{u}(x)|^{2^*_\s} dx\geq \b^*.
\end{align*}
This then ends the proof.
\end{proof}
\begin{lemma}\label{Prop-Compactness}
Let $\alpha<\b^*$ and $(u_n)_n \subset H^1_0(\O)$ be a Palais-Smale sequence for $\Psi$ at level $\alpha$. Then, up to a subsequence, there exists $u\in H^1_0(\O)$ such that
$$
\begin{cases}
u_n \to u  \qquad \textrm{ strongly in $H^1_0(\O)$}\\\\
\Psi(u)=\alpha\\\\
\Psi^\prime(u)=0.
\end{cases}
$$
\end{lemma}
\begin{proof}
Let $\alpha<\b^*$ and $(u_n)_n \subset H^1_0(\O)$ be a Palais-Smale sequence for $\Psi$ at level $\alpha$. That is
\begin{equation}\label{Torodo1}
\alpha=\frac{1}{2} \int_\O |\n u_n|^2 dx+ \frac{1}{2} \int_\O h u_n^2 dx-\frac{\l}{2^*_{s_1}} \int_\O \rho^{-s_1}_\Gamma |u_n|^{2^*_{s_1}} dx-\frac{1}{2^*_{s_2}} \int_\O \rho^{-s_2}_\Gamma |u_n|^{2^*_{s_2}} dx+o(1)
\end{equation}
and
\begin{equation}\label{Torodo2}
\int_\O \n u_n \n \varphi dx+\int_\O h u_n \varphi dx-\l \int_\O \rho^{-s_1}_\Gamma |u_n|^{2^*_{s_1}-2} u_n \varphi dx-\frac{1}{2^*_{s_2}} \int_\O \rho^{-s_2}_\Gamma |u_n|^{2^*_{s_2}-2}  u_n \varphi dx+o(1),
\end{equation}
for all $\varphi \in H^1_0(\O)$ as $n \to \infty$. Combining \eqref{Torodo1} and \eqref{Torodo2}, we obtain
\begin{equation}\label{BestConstantC}
\alpha=\left(\frac{1}{2}-\frac{1}{2^*_{s_1}}\right) \int_\O |\n u_n|^2 dx+ \left(\frac{1}{2}-\frac{1}{2^*_{s_1}}\right) \int_\O h u_n^2 dx+\left(\frac{1}{2^*_{s_1}}-\frac{1}{2^*_{s_2}} \right)\int_\O \rho^{-s_2}_\Gamma |u_n|^{2^*_{s_2}} dx+o(1).
\end{equation}
Now we use the fact that $\frac{1}{2^*_{s_1}}-\frac{1}{2^*_{s_2}}$ and $\frac{1}{2}-\frac{1}{2^*_{s_2}}$ are positive and the coercivity of the linear operator $-\D+h$, we obtain
$$
\frac{\alpha}{\left(\frac{1}{2}-\frac{1}{2^*_{s_1}}\right)}+o(1) \geq \int_\O |\n u_n|^2 dx+\int_\O h u_n^2 dx \geq  \|u_n\|_{H^1(\O)}.
$$
Consequently, up to a subsequence, there exists $u \in H^1_0(\Omega)$ such that
$
u_n
$
converges weakly to $u$ in $H^1_0(\O)$ and strongly to $L^p(\O)$ for all $2 \leq p <2^*_0$. We assume by contradiction that $u=0$. Therefore, by \eqref{BestConstantC}, we obatin
\begin{equation}\label{Amel1}
\alpha=\left(\frac{1}{2}-\frac{1}{2^*_{s_1}}\right) \int_\O |\n u_n|^2 dx+\left(\frac{1}{2^*_{s_1}}-\frac{1}{2^*_{s_2}} \right)\int_\O \rho^{-s_2}_\Gamma |u_n|^{2^*_{s_2}} dx+o(1).
\end{equation}
Moreover by Lemma \ref{OuiOui}, we get
\begin{equation}\label{Amel2}
\left(\frac{1}{2}-\frac{1}{2^*_{s_1}}\right) \int_\O |\n u_n|^2 dx+\left(\frac{1}{2^*_{s_1}}-\frac{1}{2^*_{s_2}} \right)\int_\O \rho^{-s_2}_\Gamma |u_n|^{2^*_{s_2}} dx+o(1) \geq \b^*.
\end{equation}
Hence by \eqref{Amel1} and \eqref{Amel2}, we obtain
$$
\alpha \geq \b^*,
$$
which contradicts the fact that $\alpha< \b^*$. Then $u \neq 0$ and 
$$
u_n  \to u \qquad \textrm{in $H^1_0(\O)$}.
$$
This then ends the proof.
\end{proof}
Next, we will need the following so-called mountain pass lemma due to Ambrosetti and Robinowitz, see \cite{MPL}. Then we have
\begin{lemma}\label{MPLE}\textbf{(Mountain Pass Lemma)}\\
Let $(X, \|\cdot\|_X)$ be a Banach space and $\Psi: X \to \R$ a functional of class $\calC^1$. Wa assume that
\begin{enumerate}
\item[1.] $\Psi(0)=0$;
\item[2.] There exist positive constants $A, B$ such that if $\|u\|_X=A$, then $\Psi(u)\geq B$;
\item[3.] There exists $u_0 \in X$ such that $\|u_0\| \geq A$ and $ \Psi(u_0)<B.
$
\end{enumerate}
Define
$$
\mathcal{P}= \lbrace 
\gamma \in \calC^0([0, 1]; X) \textrm{ such that } \gamma(0)=0 \textrm{ and } \gamma(1)= u_0
\rbrace.
$$
Then
$$
\beta:=\inf_{\gamma \in \mathcal{P}} \sup_{t \in [0, 1]} \Psi (\gamma(t)),
$$
is a critical value.
\end{lemma}
\begin{lemma}\label{c1A}
Let $\O$ be a bounded domain of $\R^N$, $\Gamma$ be a closed curve included in $\O$ and $h$ be a continuous function such that the linear operator is $-\D+h$ is coercive. Let $u_0 \in H^1_0(\O) \setminus \lbrace 0 \rbrace$. Then there exists $c_0$ a positive constant depending on $u_0$ and $(u_n)_n \subset H^1_0(\O)$ a a Palais-Smale sequence for $\Psi$ at level $c_0$. Moreover
$$
c_0 \leq \sup_{t \geq 0} \Psi(t u_0).
$$  
\end{lemma}
\begin{proof}
We let $ t \in \R$. Recall that for all $u \in H^1_0(\O)$, we have
$$
\Psi(t u):= \frac{t^2}{2} \int_\O (|\n u|^2+ h u^2) dx -\l\frac{|t|^{2^*_{s_1}}}{2^*_{s_1}} \int_\O \rho_\G^{-s_1} |u|^{2^*_{s_1}} dx-\frac{|t|^{2^*_{s_2}}}{2^*_{s_2}} \int_\O \rho_\G^{-s_2} |u|^{2^*_{s_2}} dx
$$
Then $\Psi \in \calC^1(H_0^1(\O), \R)$. Since $0<s_2<s_1<2$ and the fact that the function $s \longmapsto 2^*_s:=\frac{2(N-s)}{N-2}$ is decreasing, we have
\begin{equation}\label{Star1}
\lim_{t \to \infty} \Psi(t u)= -\infty.
\end{equation}
Moreover, using the fact that $2^*_{s_1}, 2^*_{s_2}>2$, then there exists sufficiently positive numbers $A, B$ such that
$$
\inf_{\|u\|=A} \Psi(u) \geq B.
$$
Therefore by the Mountain pass Lemma \ref{MPLE}, we get the desired result.
\end{proof}

\begin{proof}\textbf{of Proposition \ref{th:main12}}.\\
Let $u_0 \in H^1_0(\O)$  be a non-negative, non-vanishing function such that
$$
\sup_{t \geq 0} \Psi(t u_0)<\b^*.
$$ 
Then by Lemma \ref{c1A}, there exists $c_0>0$ depending on $u_0$ and a Palais-Smale sequence $(u_n)_n \subset H^1_0(\O)$ for $\Psi$ at level $c_0$ such that
$$
c_0 \leq \sup_{t \geq 0} \Psi(t u_0)< \b^*.
$$
By Lemma \ref{Prop-Compactness},  there exists $u \in \H^1_0(\O) \setminus \lbrace 0\rbrace$ such that, up to a subsequence, 
$$
u_n \longmapsto u \quad \textrm{strongly  in $H^1_0(\Omega)$ as $n \to \infty$ and } \Psi^\prime(u)=0. 
$$
The last equality corresponds exactly to the Euler-Lagrange equation \eqref{HHH1}. This then ends the proof.
\end{proof}
\section{Existence of solution in domains: Proof of Theorem \ref{th:main1}}\label{Section4}
Next, we let $w \in \calD^{1,2}(\R^N)$ be a positive  ground state solution of
\begin{equation}\label{BB}
-\Delta w=\l |z|^{-s_1} w^{2^*_{s_1}-1} + |z|^{-s_2} w^{2^*_{s_2}-1} \qquad \textrm{ in } \R^N
\end{equation}
and
$$
\b^*=\frac{1}{2} \int_{\R^N} |\n w|^2 dx-\frac{\l}{2^*_{s_1}} \int_{\R^N} |z|^{-s_1} |w|^{2^*_{s_1}} dx-\frac{1}{2^*_{s_2}} \int_{\R^N} |z|^{-s_2} |w|^{2^*_{s_2}} dx.
$$
In what follows, we define
$$
A_{N,s_1, s_2}:=\frac{\displaystyle\int_{\R^N} |z|^2 |\de_y w|^2 dx+\int_{\R^N} |z|^2 |\n w|^2 dx-\frac{\l}{2^*_{s_1}}\int_{\R^N} |z|^{2-s_1} |w|^{2^*_{s_1}} dx-\frac{1}{2^*_{s_2}}\int_{\R^N} |z|^{2-s_2} |w|^{2^*_{s_2}} dx}{\displaystyle 2(N-1)\int_{\R^N} w^2 dx},
$$
for $N \geq 5$ and
$
A_{4, s_1, s_2}:=3/2.
$
Then we have the following result. 
\begin{proposition}\label{Exit1}
For $N\geq 4$, we let $\Omega$ be a bounded domain of $\R^N$. We assume that
\begin{equation}\label{ExistenceAssumption}
A^N_{s_1, s_2}|\kappa(y_0)|^2+h(y_0)<0,
\end{equation}
for some positive constant. 
Then there exists $u\in H^1_0(\Omega)\setminus \lbrace 0\rbrace$ such that
$$
c^*:=\max_{t\geq 0} \Psi(t u)< \b^*.
$$
\end{proposition}
Let $\O$ a bounded domain of $\R^N$ and $\G\subset \O$ be a smooth closed curve. We let 
$\eta \in \calC^\infty_c\left(F_{y_0}\left({Q}_{2r}\right)\right)$ be such that
$$
0\leq \eta \leq 1 \qquad \textrm{ and }\qquad \eta \equiv 1 \quad \textrm{in }  \B_r .
$$
For $\e>0$, we consider the test function $u_\e: \O \to  \R$ given  by
\begin{equation}\label{eq:TestFunction-w}
u_\e(y):=\e^{\frac{2-N}{2}} \eta(F^{-1}_{y_0}(y)) w \left(\e^{-1} {F^{-1}_{y_0}(y)}  \right).
\end{equation}
In particular,  for every $x=(t,z)\in \R\times \R^{N-1}$, we have 
\begin{equation}\label{eq:TestFunction-th}
u_\e\left(F_{y_0}(x)\right):=\e^{\frac{2-N}{2}}\eta\left(x\right)\th \left( {|t|}/{\e}, {|z|}/{\e} \right).
\end{equation}
It is clear that $u_\e \in H^1_0(\O).$ Moreover, for $t\geq 0$, we have
\begin{equation}\label{PerturbedFunctional}
\Psi(t u_\e)=\frac{t^2}{2}\int_{\O} |\n u_\e|^2 +h(x) u_\e^2 dx-\l \frac{t^{2^*_{s_1}}}{2^*_{s_1}} \int_{\O} \rho_\G^{-s_1} |u_\e|^{2^*_{s_1}} dx-\frac{t^{2^*_{s_2}}}{2^*_{s_2}} \int_{\O} \rho_\G^{-s_2} |u_\e|^{2^*_{s_2}} dx.
\end{equation}
To simplify the notations, we will write $F$  in the place of $F_{y_0}$.
Recalling \eqref{eq:TestFunction-w}, we   write
$$
u_\e(y)=\e^{\frac{2-N}{2}} \eta(F^{-1} (y)) W_\e(y),
$$
where $W_\e (y) =w \left(\frac{F^{-1} (y)}{\e} \right)$.\\
\begin{lemma}\label{Lem1}
As $\e \to 0$, we have
\begin{align*}
\int_\O& |\n u_\e|^2 dy+\int_\O h(x) u_\e^2(x) dx=\int_{\R^N}|\n w|^2 dx+  \e^2\frac{|\kappa(y_0)|^2}{N-1} \int_{  \R^N} |z|^2 \left|{\de_t w} \right|^2 dx\\\\
&+\e^2\frac{|\kappa(y_0)|^2}{2(N-1)} \int_{  \R^N} |z|^2 |\n w|^2 dx+\e^2 h(y_0)\int_{\R^N} w^2 (x) dx+O\left(\e^{N-2}\right) \qquad \textrm{ for $N \geq 5$}.
\end{align*}
For $N=4$, there exists $C>0$, we have
\begin{align*}
\int_\O& |\n u_\e|^2 dy+\int_\O h(x) u_\e^2(x) dx\leq \int_{\R^N}|\n w|^2 dx+  C\e^2\left(\frac{3}{2}|\kappa(y_0)|^2+h(y_0)\right)|\ln(\e)|+O(\e^2).
\end{align*}
\end{lemma}
\begin{proof}
We have 
$$
|\n u_\e |^2 =\e^{2-N} 
\left(
\eta^2 |\n W_\e|^2+\eta^2 |\n W_\e |^2+\frac{1}{2} \n W_\e^2 \cdot \n \eta^2
\right).
$$
Then integrating by parts, we get
\begin{align}\label{eq:expan-nabla-u-eps}
\displaystyle \int_\O |\n u_\e|^2 dy &\displaystyle=\e^{2-N} \int_{F\left({Q}_{2r}\right)} \eta^2 |\n W_\e|^2  dy+\e^{2-N} \int_{F\left({Q}_{2r}\right)\setminus F\left(\B_{r}\right)} W_\e^2 \left(|\n \eta|^2-\frac{1}{2} \D \eta^2 \right)  dy \nonumber\\
&\displaystyle=\e^{2-N} \int_{F\left({Q}_{2r}\right)} \eta^2 |\n W_\e|^2  dy-\e^{2-N} \int_{F\left({Q}_{2r}\right)\setminus F\left(\B_{r}\right)}W_\e^2 \eta \D \eta  dy  \nonumber\\
&\displaystyle=\e^{2-N} \int_{F\left({Q}_{2r}\right)} \eta^2 |\n W_\e|^2  dy+O\left(\e^{2-N} \int_{F\left({Q}_{2r}\right)\setminus F\left(\B_{r}\right)} W_\e^2  dy\right) .
\end{align}
By the  change of variable $y=\frac{F(x)}{\e}$ and \eqref{eq:TestFunction-th}, we can apply Lemma \ref{lem:to-prove}, to get
\begin{align*}
&\displaystyle \int_\O |\n u_\e|^2 dy\displaystyle= \int_{ {Q}_{r/\e} } |\n w|^2_{g_\e}\sqrt{|g_\e|} dx+O\left(\e^{2} \int_{ {Q}_{2r/\e} \setminus  \B_{r/\e} }w^2 dx+
 \int_{ {Q}_{2r/\e} \setminus  \B_{r/\e} }|\n w|^2 dx \right)\\
&= \int_{\R^N}|\n w|^2 dx+   \e^2\frac{|\kappa(y_0)|^2}{N-1} \int_{  \B_{r/\e}} |z|^2 \left|{\de_t w} \right|^2 dx+\e^2\frac{|\kappa(y_0)|^2}{2(N-1)} \int_{  \B_{r/\e}} |z|^2 |\n w|^2 dx\\
&+  O\left(\e^3 \int_{ \B_{r/\e}} |x|^3 |\n w|^2 dx +\e^2 \int_{\B_{2r/\e}\setminus \B_{r/\e}}  |w|^2 dx+  \int_{\R^N\setminus\B_{r/\e}}|\n w|^2 dx+  \e^2 \int_{\B_{2r/\e}\setminus \B_{r/\e}}|z|^2|\n w|^2 dx \right).
%
\end{align*}
By Proposition \ref{Prop-Decay-Esti2}, we have, for $N \geq 4$,   that
\begin{align*}
\e^3 \int_{ \B_{r/\e}} |x|^3 |\n w|^2 dx &+\e^2 \int_{\B_{2r/\e}\setminus \B_{r/\e}}  |w|^2 dx+  \int_{\R^N\setminus\B_{r/\e}}|\n w|^2 dx\\\\
&+  \e^2 \int_{\B_{2r/\e}\setminus \B_{r/\e}}|z|^2|\n w|^2 dx=O(\e^{N-2})
\end{align*}
and 
$$
\int_{\R^N \setminus \B_{r/\e}}w^2 dx+\int_{\R^N \setminus  \B_{r/\e}} |z|^2 \left|{\de_t w} \right|^2 dx+\int_{ \R^N \setminus \B_{r/\e}} |z|^2 |\n w|^2 dx=O(\e^{N-4}) \qquad \forall N \geq 5.
$$
Therefore if $N \geq 5$, we have
\begin{align}\label{Step1}
\displaystyle\int_\O |\n u_\e|^2 dy \displaystyle
=\int_{\R^N}|\n w|^2 dx&+  \e^2\frac{|\kappa(y_0)|^2}{N-1} \int_{  \R^N} |z|^2 \left|{\de_t w} \right|^2 dx\nonumber\\\
&+\e^2\frac{|\kappa(y_0)|^2}{2(N-1)} \int_{  \R^N} |z|^2 |\n w|^2 dx+ O\left(\e^{N-2}  \right).
\end{align}
For $N=4,$ we have
\begin{align}\label{Hot1}
\displaystyle\int_\O |\n u_\e|^2 dy \displaystyle
&\leq\int_{\R^N}|\n w|^2 dx+  \e^2\frac{|\kappa(y_0)|^2}{2} \int_{  \B_{r/\e}} |z|^2 \left|\n w \right|^2 dx+O\left(\e^{2}  \right).
\end{align}
Next, by the change of variable formula $y=\frac{F(x)}{\e}$,\eqref{eq:TestFunction-th} and the continuity of the function $h$, we have
$$
\int_\O h(x) u_\e^2(x) dx= \e^2 h(y_0) \int_{Q_{r/\e}} w^2 (x) dx+ \e^2 \int_{Q_{2r/\e}\setminus Q_{r/\e}} w^2 (x) dx.
$$
Using again Proposition \ref{Prop-Decay-Esti1}, we get
$$
\int_{Q_{2r/\e}\setminus Q_{r/\e}} w^2 (x) dx=O\left(\e^{N-2}\right).
$$
Moreover for $N \geq 5$, we have
$$
\int_{\R^N\setminus Q_{r/\e}} w^2 (x) dx=O\left(\e^{N-2}\right)
$$
Therefore
\begin{equation}\label{Step21}
\int_\O u_\e^2(x) dx=\e^2 h(y_0)\int_{\R^N} w^2 (x) dx+o\left(\e^2\right).
\end{equation}
If $N=4$, we have
\begin{equation}\label{Step210}
\int_\O u_\e^2(x) dx=\e^2 h(y_0)\int_{Q_{r/\e}} w^2 (x) dx+O\left(\e^2\right).
\end{equation}
Next, we assume that $N=4$ and we let $\eta_\e(x)=\eta(\e x)$. 
We multiply \eqref{BB} by $|z|^2\eta_\e w$ and integrate by parts to get
\begin{align*}
\displaystyle & \l \int_{Q_{2r/\e}}    \eta_\e  |z|^{2-s_1}  w^{2^*_{s_1}} dx+\int_{Q_{2r/\e}}    \eta_\e  |z|^{2-s_2}  w^{2^*_{s_2}} dx\displaystyle=\int_{Q_{2r/\e}} \n w \cdot \n \left(\eta_\e  |z|^2 w\right) dx\\
&\displaystyle= \int_{Q_{2r/\e}} \eta_\e  |z|^2 |\n w|^2 dx+\frac{1}{2} \int_{Q_{2r/\e}} \n w^2 \cdot \n \left(|z|^2\eta_\e \right) dx \int_{Q_{2r/\e}} \eta_\e  |z|^2 |\n w|^2 dx-\frac{1}{2} \int_{Q_{2r/\e}} w^2 \D\left(|z|^2\eta_\e \right) dx\\
&\displaystyle=\int_{Q_{2r/\e}} \eta_\e  |z|^2 |\n w|^2 dx-3\int_{Q_{2r/\e}} w^2 \eta_\e      dx =\displaystyle\quad-\frac{1}{2} \int_{Q_{2r/\e}\setminus Q_{r/\e}} w^2 (|z|^2\D\eta_\e+ 4\n \eta_\e\cdot z) dx.
\end{align*} 
We then deduce that 
\begin{align*}
\l \int_{Q_{2r/\e}} |z|^{2-s_1}  w^{2^*_{s_1}} dx&+\int_{Q_{2r/\e}} |z|^{2-s_2}  w^{2^*_{s_2}} dx= \int_{Q_{r/\e}}  |z|^2 |\n w|^2 dx-(N-1)\int_{Q_{r/\e}} w^2        dx\\
&+O\left( \int_{Q_{2r/\e}\setminus Q_{r/\e}}      |z|^{2-\s}  w^{2^*_\s} dx+  \int_{Q_{2r/\e}\setminus  Q_{r/\e}}  |z|^2 |\n w|^2 dx+ \int_{Q_{2r/\e}\setminus  Q_{r/\e}}     w^2 dx     \right)\\
&+ O\left(  \e \int_{Q_{2r/\e}\setminus  Q_{r/\e}}  |z|  |\n w| dx+ \e^2 \int_{Q_{2r/\e}\setminus  Q_{r/\e}}   |z|^2   w^2 dx     \right).
\end{align*} 
By Proposition \ref{Prop-Decay-Esti1}, we have
$$
\l \int_{Q_{2r/\e}} |z|^{2-s_1}  w^{2^*_{s_1}} dx+\int_{Q_{2r/\e}} |z|^{2-s_2}  w^{2^*_{s_2}} dx=O(1)
$$
and 
\begin{align*}
+\int_{Q_{2r/\e}\setminus Q_{r/\e}}      |z|^{2-\s}  w^{2^*_\s} dx&+  \int_{Q_{2r/\e}\setminus  Q_{r/\e}}  |z|^2 |\n w|^2 dx+ \int_{Q_{2r/\e}\setminus  Q_{r/\e}}     w^2 dx \\
&+\e \int_{Q_{2r/\e}\setminus  Q_{r/\e}}  |z|  |\n w| dx+ \e^2 \int_{Q_{2r/\e}\setminus  Q_{r/\e}}   |z|^2   w^2 dx=O(\e^2).
\end{align*}
Therefore
 \begin{align}\label{Hot22}
 \int_{Q_{r/\e}}  |z|^2 |\n w|^2 dx=3\int_{Q_{r/\e}} w^2dx +O(1).
\end{align} 
To finish, we use Proposition \ref{Prop-Decay-Esti1} to get
\begin{align}\label{Hot2222}
\int_{Q_{r/\e}} w^2dx  \leq C  \int_{Q_{r/\e}}\frac{dx}{1+|x|^4}=C|S^3| \int_0^{r\e} \frac{t^3 dt}{1+t^4} \leq C(1+|\ln(\e)|),
\end{align}
where $C$ is a positive constant that may change from an inequality to another.
Thus the result follows  immediately from  \eqref{Step1}, \eqref{Hot1}, \eqref{Step21}, \eqref{Step210}  \eqref{Hot22} and  \eqref{Hot2222}. This then ends the proof.
\end{proof}

\begin{lemma}\label{Lem2}
Let $s \in (0, 2)$. Then we have
$$
\int_\O \rho^{-s}_\G |u_\e|^{2^*_s} dx =\int_{\R^N} |z|^{-s} w^{2^*_s} dx+
\e^2\frac{|\kappa(y_0)|^2}{2(N-1)} \int_{\R^N} |z|^{2-s} w^{2^*_s} dx+O\left(\e^{N-s}\right).
$$
\end{lemma}
\begin{proof}
Let $s\in [0, 2)$. Then by the change of variable $y=\frac{F(x)}{\e}$, \eqref{eq:rho_Gamm-is-mod-z} and   \eqref{DeterminantMetric},  we get 
\begin{align*}
&\displaystyle\int_\O \rho^{-s}_\G |u_\e|^{2^*_s} dy=\int_{\B_{r/\e}} |z|^{-s} w^{2^*_s} \sqrt{|g_\e|}dx+ O\left( \int_{\B_{2r/\e}\setminus \B_{r/\e}}  |z|^{-s} (\eta(\e x) w)^{2^*_s} dx \right)\\
%
%
&\displaystyle= \int_{\B_{r/\e}} |z|^{-s} w^{2^*_s} dx+\e^2\frac{|\kappa(y_0)|^2}{2(N-1)} \int_{\B_{r/\e}} |z|^{2-s} w^{2^*_s} dx\\\
&+O\left(\e^3 \int_{\B_{r/\e}} |x|^3 |z|^{-s} w^{2^*_s} dx+   \int_{\B_{2r/\e}\setminus \B_{r/\e}}  |z|^{-s}  w^{2^*_s} dx \right)\\
&\displaystyle= \int_{\R^N} |z|^{-s} w^{2^*_s} dx+\e^2\frac{|\kappa(y_0)|^2}{2(N-1)} \int_{\B_{r/\e}} |z|^{2-s} w^{2^*_s} dx\\
&\quad +O\left(\e^3 \int_{\B_{r/\e}} |x|^3 |z|^{-s} w^{2^*_s} dx + \int_{\R^N \setminus \B_{r/\e}} |z|^{-s} w^{2^*_s} dx+ \int_{\B_{2r/\e}\setminus \B_{r/\e}}  |z|^{-s}  w^{2^*_s} dx\right).
\end{align*}
By Proposition \ref{Prop-Decay-Esti1}, we have
$$
\e^3 \int_{\B_{r/\e}} |x|^3 |z|^{-s} w^{2^*_s} dx+\int_{\R^N \setminus \B_{r/\e}} |z|^{-s} w^{2^*_s} dx+\int_{\B_{2r/\e}\setminus \B_{r/\e}}  |z|^{-s}  w^{2^*_s} dx=O\left(\e^{N-s}\right)
$$
and
\begin{equation}\label{Step3A}
\int_{\R^N \setminus \B_{r/\e}} |z|^{2-s} w^{2^*_{s}} dx=O\left(\e^{N-2-s}\right) \qquad \forall N \geq 4.
\end{equation}
Therefore
\begin{equation}\label{Step3}
\int_\O \rho^{-s}_\G |u_\e|^{2^*_s} dx =\int_{\R^N} |z|^{-s} w^{2^*_s} dx+
\e^2\frac{|\kappa(y_0)|^2}{2(N-1)} \int_{\R^N} |z|^{2-s} w^{2^*_s} dx+O\left(\e^{N-s}\right),
\end{equation}
as $\e \to 0$.  This then ends the proof.
\end{proof}
Now we are in position to prove Proposition \ref{Exit1}. 
\begin{proof} \textbf{of Proposition \ref{Exit1}}\\
Recall that, for all $t \geq 0$ and all $u \in H^1_0(\O)$, we have
$$
\Psi(tu):=\frac{t^2}{2} \int_\O |\n u|^2 dx+\frac{1}{2} \int_\O h(x) u^2 dx-t^{2^*_{s_1}}\frac{\l}{2^*_{s_1}} \int_\O \frac{|u|^{2^*_{s_1}}}{\rho_\G^{s_1}(x)} dx-t^{2^*_{s_2}}\frac{1}{2^*_{s_2}} \int_\O \frac{|u|^{2^*_{s_2}}}{\rho_\G^{s_2}(x)} dx.
$$
Then by Lemma \ref{Lem1} and Lemma \ref{Lem2}, we have
\begin{align*} 
J\left(t u_\e\right)=\Psi(t w)&+\e^2t^2 \frac{|\kappa(y_0)|^2}{2(N-1)} \left(\int_{  \B_{r/\e}} |z|^2 \left|{\de_t w} \right|^2 dx+\int_{  \B_{r/\e}} |z|^2 |\n w|^2 dx\right)\\\
&+\e^2 t^2 h(y_0)\int_{\B_{r/\e}}w^2 dx+\e^2 \l \frac{t^{2^*_{s_1}}}{2^*_{s_1}} \frac{|\kappa(y_0)|^2}{2(N-1)}\int_{\B_{r/\e}} |z|^{2-s_1} w^{2^*_{s_1}} dx\\\\
&+\e^2 \frac{t^{2^*_{s_2}}}{2^*_{s_2}} \frac{|\kappa(y_0)|^2}{2(N-1)}\int_{\B_{r/\e}} |z|^{2-s_2} w^{2^*_{s_2}} dx+O\left(\e^{N-2}\right) \qquad \textrm{ for $N\geq 5$}.
\end{align*}
For $N=4$, there exists $C>0$, we have
\begin{align*}
J(t u_\e)\leq \Psi(t w)+  C\e^2 t^2 \left(\frac{3}{2}|\kappa(y_0)|^2+h(y_0)\right)|\ln(\e)|+O(\e^2).
\end{align*}
Since $2^*_{s_2}> 2^*_{s_1}$, $J(tu_\e)$ has a unique maximum, we have
$$
\max_{t\geq 0} \Psi(tw)=\Psi(w)=\b^*.
$$
Therefore, the maximum of $J(t u_\e)$ occurs at $t_\e:=1+o_\e(1)$. Next setting
\begin{align*}
\mathcal{G}(t w):&=\displaystyle\e^2t^2 \frac{|\kappa(y_0)|^2}{2(N-1)} \left(\int_{\R^N} |z|^2 \left|{\de_t w} \right|^2 dx+\int_{\R^N} |z|^2 |\n w|^2 dx\right)\\\\
&+\e^2 t^2 h(y_0)\int_{\R^N}w^2 dx+\e^2 \displaystyle \l \frac{t^{2^*_{s_1}}}{2^*_{s_1}} \frac{|\kappa(y_0)|^2}{2(N-1)}\int_{\R^N} |z|^{2-s_1} w^{2^*_{s_1}} dx\\\\
&+\e^2 \frac{t^{2^*_{s_2}}}{2^*_{s_2}} \frac{|\kappa(y_0)|^2}{2(N-1)}\int_{\R^N} |z|^{2-s_2} w^{2^*_{s_2}} dx+o(\e^2) \quad\textrm{ for $N \geq 5$},
\end{align*}
and
$$
\mathcal{G}(t w)=C\e^2 |\ln(\e)| t^2 \left(\frac{3}{2}|\kappa(y_0)|^2+h(y_0)\right) +O(\e^2) \qquad \textrm{ for $N=4$}.
$$
Thanks to assumption \eqref{ExistenceAssumption}, we have
$$
\mathcal{G}(w)<0.
$$
Therefore
$$
\max_{t \geq 0} J(t u_\e):= J(t_\e u_\e)\leq \Psi(t_\e w)+\e^2 \mathcal{G}(t_\e w) \leq \Psi(t_\e w) < \Psi(w)=\b^*
$$

We thus get the desired result.
\end{proof}
\begin{proof}\textbf{of Theorem \ref{th:main1}}
The proof of Theorem \ref{th:main1} is a direct consequence of  Proposition \ref{th:main12} and Proposition \ref{Exit1}.
\end{proof}
\vspace{0.5cm}
%
%
%
\textbf{Declarations}\\\

\textbf{Conflict of interest statement:} The corresponding author states that there is no conflict of interest.\\\

\textbf{Code availability statement:} not Applicable.\\\

\textbf{Data availability statement:} Data sharing not applicable to this article as no datasets were generated or analysed during the current study.


\begin{thebibliography}
\footnotesize
%
\bibitem{MPL} A. Ambrosetti and P. H. Rabinowitz, {Dual variational methods in critical point theory and applications},  Journal of functional Analysis 14, no. 4 (1973): 349-381.\\
\bibitem{AAZ} A. D. Alexandrov, A characteristic property of the spheres. Ann. Mat. Pura Appl.,  58  (1962),  303-315.\\
\bibitem{BNAZ}H. Berestycki and L. Nirenberg, On the method of moving planes and the sliding method. Bol. Soc. Bras. Mat.,  22  (1991),  1-37.\\
\bibitem{BPZ} T. Bartsch, S. Peng and Z. Zhang, {\it  Existence and non-existence of solutions to elliptic equations related to the Caffarelli-Kohn-Nirenberg inequalities}, Calc. Var. Partial Differential Equations \textbf{30}(2007), no. 1, 113-136.\\
\bibitem{CH} D. Cao and P. Han, {\it  Solutions to critical equation with multi-singular inverse square potentials}, J. Differential Equations \textbf{224} (2006), no. 2, 332-372.\\
\bibitem{CLAZ} J-L. Chern and   C-S Lin, {\it Minimizers of Caffarelli-Kohn-Nirenberg Inequalities with the singularity on the boundary}.   Arch. Rational Mech. Anal.    197, Number 2 (2010), 401-432.\\
%
%
%
\bibitem{DNAZ}  A. V. Demyanov, I.A. Nazarov, {\it On the solvability of the Dirichlet problem for the semilinear Schr\"{o}dinger equation with a singular potential}, (Russian)  Zap. Nauchn. Sem. S.-Peterburg. Otdel. Mat. Inst. Steklov, (POMI)  336  (2006),  Kraev. Zadachi Mat. Fiz.i Smezh. Vopr. Teor. Funkts. 37, 25--45, 274;  translation in  J. Math. Sci. (N.Y.) 143  (2007),  no. 2, 2857--2868.\\
%
%
%
%
\bibitem{CW} F. Catrina and Z. Q. Wang, {\it On the Caffarelli-Kohn-Nirenberg inequalities: sharp constants, existence (and nonexistence) and symmetry of extremal functions}, Comm. Pure Appl. Math \textbf{54} (2001), no. 2, 229-258.\\
\bibitem{CKN} L. Caffarelli, R. Kohn and L. Nirenberg, {\it First order interpolation inequalities with weights}, Compositio Math \textbf{53}(1984), no. 3, 332-372.\\
\bibitem{CL} J. L. Chern and C.S. Lin,{\it  Minimizers of Caffarelli-Kohn-Nirenberg inequalities on domains with the singularity on the boundary} , Arch. Rational Mech. Anal. 197 (2010), 401–432.\\
\bibitem{CC} K. S. Chou and C. W. Chu, {\it On the best constant for a weighted Sobolev-Hardy inequality}, J. London Math. Soc. (2) 48 (1993), no. 1, 137–151.\\
\bibitem{DN}  A. V. Demyanov,  and A. I. I.  Nazarov, {\it On solvability of Dirichlet problem to semilinear Schrodinger equation with singular potential},  Zapiski Nauchnykh Seminarov POMI, 336(2006), 25-45.\\
\bibitem{DET} J. Dolbeault, M. J. Esteban and G. Tarantello, {\it The role of Onofri type inequalities in
the symmetry properties of extremals for Caffarelli-Kohn-Nirenberg inequalities, in two
space dimensions} , Ann. Sc. Norm. Super. Pisa Cl. Sci. (5) 7 (2008), no. 2, 313–341.\\
\bibitem{FMT}M. M. Fall, I. A. Minlend and E. H. A. Thiam, {\it The role of the mean curvature in a Hardy-Sobolev trace inequality}. NoDEA Nonlinear Differential Equations Appl. 22(2015), no.5, 1047-1066.\\
\bibitem{GK} N. Ghoussoub and X. S. Kang, {\it Hardy-Sobolev critical elliptic equations with boundary
singularities}, Ann. Inst. H. Poincar´e Anal. Non Linéaire 21 (2004), no.6, 767–793.\\
\bibitem{GR} N. Ghoussoub and F. Robert, {\it The effect of curvature on the best constant in the Hardy-
Sobolev inequalities}, Geom. Funct. Anal. 16 (2006), no.6, 1201–1245.\\
\bibitem{GR2} N. Ghoussoub and F. Robert, {\it On the Hardy-Schr\"odinger operator with a boundary singularity}, Preprint 2014. https://arxiv.org/abs/1410.1913.\\
%
\bibitem{GR3} N. Ghoussoub and F. Robert,  {\it Sobolev inequalities for the Hardy-Schr\"odinger operator: extremals and critical dimensions}, Bull. Math. Sci. 6 (2016), no. 1, 89-144.\\
%

\bibitem{GR4} N. Ghoussoub and F. Robert,  { \it  Elliptic Equations with Critical Growth and a Large Set of Boundary Singularities}. Trans.   Amer. Math. Soc., Vol. 361, No. 9 (Sep., 2009), pp. 4843-4870.\\
%
\bibitem{GR51} N. Ghoussoub and F. Robert, { \it Concentration estimates for Emden-Fowler equations with boundary singularities and critical growth}. IMRP Int. Math. Res. Pap. (2006), 21867, 1-85.\\
\bibitem{FMS} I. Fabbri, G. Mancini and K. Sandeep, {\it Classification of solutions of a critical Hardy–Sobolev operator}, Journal of Differential Equations 224.2 (2006): 258-276.\\
\bibitem{Fall-Thiam} M. M. Fall and E. H. A. Thiam, {\it Hardy-Sobolev inequality with singularity a curve}, (2018): 151-181.\\
\bibitem{Li} Y. Y. Li, {\it Prescribing scalar curvature on $S^n$ and related problems}, I, J. Differential Equations 120 (1995), no. 2, 319–410.\\
\bibitem{GY} N. Ghoussoub and C. Yuan, {\it  Multiple solutions for quasi-linear PDEs involving the critical Sobolev and Hardy exponents}, Transactions of the American Mathematical Society, 2000, 352(12), 5703-5743.\\
\bibitem{HLW} C. H. Hsia, C.S. Lin and  H. Wadade, {\it Revisiting an idea of Brézis and Nirenberg}, Journal of Functional Analysis, 259(7)(2010), 1816-1849.\\
\bibitem{Jaber1} H. Jaber, {\it Hardy-Sobolev equations on compact Riemannian manifolds}, Nonlinear Anal. 103(2014), 39-54.\\
\bibitem{Lieb} E. H. Lieb, {\it Sharp constants in the Hardy-Littlewood-Sobolev and related inequalities},
Ann. of Math. 118 (1983), no.2, 349–374.\\
\bibitem{LLin} Y. Li and C Lin, {\it A nonlinear elliptic PDE with two Sobolev–Hardy critical exponents}, Archive for Rational Mechanics and Analysis 203.3 (2012): 943-968.\\
\bibitem{Lin} C. S. Lin, {\it Interpolation inequalities with weights}, Comm. Partial Differential Equations 11 (1986), no. 14, 1515–1538.\\
\bibitem{LW} C. S. Lin and Z. Q. Wang, {\it Symmetry of extremal functions for the Caffarelli-Kohn-
Nirenberg inequalities}, Proc. Amer. Math. Soc. 132 (2004), no.6, 1685–1691.\\
\bibitem{Musina} R. Musina, {\it Ground state solutions of a critical problem involving cylinderical weights}, Nonlinear Anal. 68 (2008), no. 12, 3972–3986.\\
\bibitem{Struwe} M. Struwe, {\it Variational Methods: Applications to nonlinear Partial Differential Equations and Hamiltonian Systems}, Springer Science and Business Media, 2008, Vol. 34.\\
\bibitem{SAZ}J. Serrin, A symmetry theorem in potential theory. Arch. Rational Mech. Anal.,  43  (1971), pp. 304-318.\\
\bibitem{Thiam-Hardy} E. H. A. Thiam, {\it Hardy and Hardy-Sobolev Inequalities on Riemannian manifolds}, Imhotep Mathematical Journal Volume 2, Numéro 1, (2017), pp. 14-35.\\
\bibitem{Th1} E. H. A. Thiam, {\it Hardy-Sobolev inequality with higher dimensional singularities}, Analysis (Berlin) 39 (2019), no. 3, 79–96.\\
\bibitem{Thiam-Mass-Effect}, E. H. A. Thiam, {Mass effect on an elliptic PDE involving two Hardy-Sobolev critical exponents}, Preprint.
\end{thebibliography}
\end{document}